\documentclass[11pt,reqno]{amsart}
\usepackage[american]{babel}
\usepackage{comment}
\usepackage{amsmath}
\usepackage{dsfont,mathtools,amssymb}
\usepackage{color}
\usepackage[hidelinks]{hyperref}

\mathtoolsset{showonlyrefs,showmanualtags}

\usepackage[toc]{appendix}

\newtheorem{theorem}{Theorem}[section]
\newtheorem{proposition}[theorem]{Proposition}
\newtheorem{lemma}[theorem]{Lemma}
\newtheorem{corollary}[theorem]{Corollary}
\theoremstyle{remark}
\newtheorem{remark}[theorem]{Remark}
\theoremstyle{definition}

\newtheorem*{acknowledgement}{Acknowledgement}
\numberwithin{equation}{section}
\numberwithin{figure}{section}

\newcommand{\E}{\mathds{E}}
\renewcommand{\P}{\mathds{P}}

\newcommand{\R}{\mathbb{R}}

\newcommand{\N}{\mathbb{N}}

\newcommand{\dd}{{\rm d}}
\newcommand{\fstop}{\; \text{.}}
\newcommand{\comma}{\; \text{,}\;\;}

\newcommand{\tonde}[1]{\left(#1\right)}
\newcommand{\quadre}[1]{\left[#1\right]}

\newcommand{\abs}[1]{\left\lvert#1\right\rvert}

\newcommand{\eqdef}{\coloneqq}

\newcommand{\car}{\mathds{1}}

\newcommand{\scalar}[2]{\left\langle #1, #2 \right\rangle}

\newcommand{\var}{{\rm Var}}

\newcommand{\cW}{\ensuremath{\mathcal W}}

\newcommand{\set}[1]{\left\{#1\right\}}							%Insieme, graffe

\let\a=\alpha \let\b=\beta   \let\d=\delta  
 \let\g=\gamma       \let\l=\lambda
            
\let\r=\rho      
  
\let\D=\Delta

\subjclass[2020]{Primary 60K35; Secondary 82B20; 82C26}

\keywords{Mixing times of Markov chains; cutoff phenomenon; exchange models} 

\author{Pietro Caputo, Matteo Quattropani, Federico Sau}
\address{Pietro Caputo\\ Universit\`{a} degli Studi Roma Tre}
\email{pietro.caputo@uniroma3.it}
\address{Matteo Quattropani\\ Universit\`{a} degli Studi Roma Tre}
\email{matteo.quattropani@uniroma3.it}
\address{Federico Sau\\ Universit\`{a} degli Studi di Milano}
\email{federico.sau@unimi.it}

\begin{document}
		  	\title[A universal cutoff phenomenon for 
		  	mean-field
		  	exchange models]{A universal cutoff phenomenon\\ for 
		  	mean-field
		  	exchange models}
	
\begin{abstract} 
We study a broad class of high-dimensional mean-field exchange models, encompassing both noisy and singular dynamics, along with their dual processes. This includes a generalized version of the averaging process as well as some non-reversible extensions of classical exchange dynamics, such as the flat Kac model. Within a unified framework, we analyze convergence to stationarity from worst-case initial data in Wasserstein distance. Our main result establishes a universal cutoff phenomenon at an explicit mixing time, with a precise window and limiting Gaussian profile. The mixing time and profile are  characterized in terms of the logarithm of the size-biased redistribution random variable, thus admitting a natural entropic interpretation. 
\end{abstract}
	
	\maketitle
	
%	\setcounter{tocdepth}{4}
%	\tableofcontents
	\thispagestyle{empty}

		\section{Introduction}
		Stochastic exchange models play a fundamental role across diverse fields, from statistical physics to the social sciences. In statistical physics, they are typically interpreted as models of mass redistribution or energy exchange, while in social sciences they arise in contexts such as wealth redistribution, opinion dynamics, and information spreading; see, e.g., \cite{kipnis_heat_1982,castellano_statistical_2009,lanchier_stochastic_2024} and references therein. 
		
In kinetic theory, following the foundational program of Mark Kac \cite{kac_foundations_1956}, mean-field stochastic exchange models have been widely used as particle approximations to nonlinear equations, particularly through the framework of propagation of chaos. Over the past two decades, significant progress has been made in analyzing convergence to stationarity for such models using spectral and entropy-based methods (see, e.g., \cite{carlen_carvalho_loss_determination_2003, villani_cercignani_2003, carlen_entropy_2010, mischler_mouhot_kac_2013}). These results, however, rely on the assumption that the initial condition is sufficiently regular, typically absolutely continuous with respect to the equilibrium measure. In contrast, the analysis of convergence from singular initial data, such as Dirac-type configurations that often represent worst-case scenarios, remains far less developed.

In this paper, we consider three closely related families of mean-field exchange models and investigate their convergence to stationarity when initialized from a Dirac mass. Our main result establishes a universal cutoff phenomenon in a natural Wasserstein distance, under remarkably mild assumptions.

The study of cutoff phenomena---sharp transitions in the convergence to equilibrium---has emerged as a central topic in the modern theory of stochastic processes. While substantial breakthroughs have been achieved in the setting of finite-state Markov chains and in certain classes of diffusion processes \cite{lubetzky2010cutoff,lacoin2016mixing,bordenave_cutoff2019,salez_cutoff_2024, salez2025cutoff}, a comprehensive theory remains elusive. In particular, the case of jump processes with continuous state spaces---such as the exchange models considered here---has received comparatively little attention.

In this direction, some partial results are known for the flat Kac model, where the correct order of the worst-case total variation mixing time was determined in \cite{smith_gibbs_2014}. Additional results in one-dimensional settings appear in \cite{randall_winkler_mixing_2005, caputo_mixing_2019,labbe_petit_hydrodynamic_2025}. Despite further refinements in \cite{pillai_smith_kac_2017, pillai_smith_mixing_2018}, however, existing techniques do not yield cutoff results for the mean-field setup, nor do they extend to the broader classes of dynamics addressed in the present work. It is also worth emphasizing that in continuous settings, total variation distance can be overly sensitive to microscopic irregularities---especially when equilibrium measures are singular---thereby motivating the use of weaker metrics. In this context, Wasserstein distances offer a natural and robust alternative \cite{aldous_lecture_2012, banerjee2020rates, quattropani2021mixing, oliveira_convergence_2009}.

\subsection{Stochastic exchange models}
		
The models to be considered are discrete-time Markov chains described as follows. 
For a fixed $n\in\N$, an energy configuration of $n$ labeled particles is given by a vector $\eta=(\eta(1),\ldots, \eta(n))\in \R_+^n$, where the (nonnegative) scalar $\eta(x)$, $x=1,\dots,n$, is interpreted as the energy of the $x$-th particle. The state space is the set of all possible such configurations. At each time step the current configuration $\eta$ is updated by  
choosing an ordered pair $(x,y)$ of distinct particles selected uniformly at random, and an energy exchange between the two particles is performed.  We consider three different types of exchange mechanisms:
\begin{itemize}
	\item \emph{Stochastic Redistribution Model {(SRM)}}: \[(\eta(x),\eta(y))\mapsto (X(\eta(x)+\eta(y)), (1-X)(\eta(x)+\eta(y)))\fstop\]
	\item \emph{Stochastic Equalization  Model {(SEM)}}: \[(\eta(x),\eta(y))\mapsto (X\eta(x)+(1-X)\eta(y), X\eta(x)+(1-X)\eta(y))\fstop\]
	\item \emph{Generalized Averaging Model {(GAM)}}: \[(\eta(x),\eta(y))\mapsto (X\eta(x)+(1-X)\eta(y), (1-X)\eta(x)+X\eta(y))\fstop\]
	\end{itemize}
In all three cases, $X$ is a random variable taking values in $[0,1]$, representing a fraction of the energy to be redistributed. The only assumptions on the variable $X$ are:
\begin{itemize}
	\item \emph{Symmetry:} $X$ and $1-X$ have the same distribution;
	\item \emph{Non-degeneracy:} $X$ is not a Bernoulli random variable, i.e., $	\P(X\not\in\{0,1\})>0$.
\end{itemize}
Thus, if $\eta_0\in \R_+^n$ denotes the initial configuration, the configuration $\eta_t$ at time $t\in\N$ is given by 
\begin{equation}\label{eq:repre}
	\eta_t = \eta_0 R_1\cdots R_t\comma\qquad t \ge 1,
\end{equation}
  where $R_i$ are i.i.d.\ $n\times n$ random matrices obtained as follows. Letting $(x,y)$ denote the random pair, and letting $X$ denote the redistribution random variable chosen at time $t$, $R_t$ is the block diagonal matrix given by the identity except for the $2\times 2$ block identified by the pair $(x,y)$, where it takes, in SRM, SEM and GAM respectively, the form
\begin{equation}\label{eq:repre1}
r=\begin{pmatrix} X & 1-X\\
		X & 1-X
	\end{pmatrix}\,,\qquad r^{\top}=\begin{pmatrix} X & X\\
		1-X & 1-X
	\end{pmatrix}\,,\qquad q=\begin{pmatrix} X & 1-X\\
		1-X & X
	\end{pmatrix}\,.
\end{equation}
We note that, while SRM and GAM both conserve the total energy, SEM does not. However, it satisfies the maximum principle 
\begin{equation}\label{eq:max}
\max_{x}\eta_t(x)\le \max_{x}\eta_{0}(x)\,,
\end{equation} for all $t\ge 1$, and thus, given an initial configuration $\eta_0\in \R_+^n$, all three models can be thought of as evolving in a compact subset of $\R_+^n$. Without loss of generality, we assume that the initial configuration $\eta_0$ is in $\D$, where $\D$ denotes the $n$-simplex, that is, the set of $\eta\in\R_+^n$ such that $\textstyle{\sum_{x}}\,\eta(x)=1$.
Clearly, the three models coincide if $X\equiv 1/2$ is a point mass. In this special case, they all reduce to the mean-field averaging model recently studied in \cite{chatterjee2020phase}. We shall see that, given an initial energy vector $\eta_0\in \D$,   all three models will reach a limiting stationary state as $t\to\infty$. 
\begin{proposition}\label{prop:uniqueness}
For  any $\eta_0\in \D$, in all three models,  $\eta_t$ converges weakly to a random variable $\eta_\infty$ as $t\to\infty$. The GAM has a degenerate limit $\eta_\infty(x)\equiv\frac1n$, while the SEM has a random degenerate limit $\eta_\infty(x)\equiv Y$, where $Y\in[0,1]$ is a random variable whose distribution depends on the initial state $\eta_0$.  The SRM has a unique (non-degenerate, unless $X\equiv 1/2$) limit.
\end{proposition}
The stationary state of the SRM and SEM is known explicitly only in the case where $X$ is the ${\rm Beta}(\alpha,\alpha)$ random variable. 

\begin{lemma}\label{lem:gamma}
Fix $\a>0$ and assume $X\sim{\rm Beta}(\alpha,\alpha)$. Then the SRM is a reversible Markov chain and its unique stationary state is the Dirichlet distribution obtained as the product of i.i.d.\  ${\rm Gamma}(\alpha,1)$ random variables conditioned to have total sum equal to $1$. Moreover, taking $\eta_0=\d_{x_0}$, the Dirac mass at  $x_0$, the SEM has $\eta_\infty(x)\equiv Y$, where $Y\sim {\rm Beta}(\alpha,\alpha(n-1))$.	
\end{lemma}
 In the special case where $X$ is the uniform distribution (i.e., $X\sim {\rm Beta}(\alpha,\alpha)$ with $\alpha=1$), the SRM is the mean-field version of the Kipnis-Marchioro-Presutti model \cite{kipnis_heat_1982} introduced for the study of heat flow in a chain of oscillators; the SRM in this context is also known as \textquotedblleft flat Kac model\textquotedblright\ \cite{caputo_kac2008} or \textquotedblleft Gibbs sampler on the simplex\textquotedblright\ \cite{smith_gibbs_2014}. The SEM admits a natural interpretation as a dual process of the SRM. In the special case of Beta distributed $X$, the SEM has appeared in the literature under the name of ``hidden parameter (or temperature) model'' \cite{de_masi_ferrari_gabrielli_hidden_2023,giardina_redig_tol_intertwining_2024}. Generalizations were recently analyzed in \cite{kim2025spectral}.

\subsection{Main results}
Because of the degeneracy in the limiting random variable $\eta_\infty$ for the SEM and GAM, the total variation distance between $\eta_t$ and $\eta_\infty$ does not converge to zero for these models. Therefore, we consider the weaker $1$-Wasserstein distance. 
 For all integers 	$t\ge 0$, we write
	\begin{align}\label{def:W1-dritto}
	W_1(\eta_0,t)&:=	W_1({\rm Law}(\eta_t),{\rm Law}(\eta_\infty))= \inf_{(\xi,\vartheta)\sim(\eta_t,\eta_{\infty})}\E\big[\|\xi-\vartheta\|_1\big]\,,
\end{align}
where $\|v\|_1=\sum_x|v(x)|$ is the $\ell_1$-norm on $\R^n$, the infimum is over all couplings $(\xi,\vartheta)$ of the random variables $\eta_t$ and $\eta_\infty$, and the expectation is taken with respect to such couplings. 
Because of the mean-field character of our models, it is also natural 
to consider a metrics which is not sensitive to  permutations of the labels of the particles. Therefore, we introduce the distance
\begin{align}\label{def:W-storto}
	\cW_1(\eta_0,t)&:= \inf_{(\xi,\vartheta)\sim(\eta_t,\eta_{\infty})}\E\Big[\inf_{\sigma \in \mathfrak S_n} \|\xi-\vartheta\circ\sigma\|_1\Big]\,,
\end{align}
 where $\mathfrak S_n$ stands for the set of permutations on $n$ elements and $\vartheta\circ\sigma$ is the energy configuration after permuting the particles' labels according to $\sigma$. Equivalently, one can also write $ \cW_1(\eta_0,t)$ as the $1$-Wasserstein distance between the empirical distributions on $\R$ of $\eta_t$ and $\eta_{\infty}$.  
Clearly, 
\begin{equation}\label{eq:comparison2}
	\cW_1(\eta_0,t)\le W_1(\eta_{0},t)\comma
	\qquad t \ge 0\,.
\end{equation}
Actually, for SEM and GAM, since $\eta_\infty(x)$ is independent of $x$, the inequality in \eqref{eq:comparison2} saturates to an equality.
We mention a simple monotonicity property of these distances.

\begin{lemma}\label{lemma:monotone}
For any  $\eta_0\in\D$, in all three models 
the distances in \eqref{def:W1-dritto} and \eqref{def:W-storto} are non-increasing in $t\ge 0$. 
\end{lemma}
		Since we are interested in worst-case initializations, we define 
	\begin{equation}\label{def:worst}
W_1(t):= \max_{\eta_0\in\D}W_1(\eta_0,t)\,,\qquad 	\cW_1(t):= \max_{\eta_0\in\D}\cW_1(\eta_0,t)\,.
\end{equation}
Our main result, Theorem \ref{theorem} below, shows that for all three models, both distances	$W_1(t)$ and $\cW_1(t)$ display a cutoff phenomenon at a time of order $n\log n$ with a time window of order  $n\sqrt{\log n}$. Moreover, we obtain an explicit Gaussian cutoff profile, with the relevant constants being uniquely characterized in terms of the first two moments of the logarithm of the size-biased distribution of the variable $X$. 
	
	\begin{remark}\label{rem:worst}
	While for the GAM it is easily seen that $W_1(\eta_0,t)$ and $\cW_1(\eta_0,t)$ are maximized when $\eta_0=\d_{x_0}$ is a Dirac mass at  $x_0$, the analogous property for SRM and SEM does not seem to be straightforward. Our results will however show that this is the case in the sense that $W_1(t)$ and $W_1(\d_{x_0},t)$ (as well as $\cW_1(t)$ and $\cW_1(\d_{x_0},t)$) satisfy the very same cutoff phenomenon for all three models.
	\end{remark}  
	 
	To state our main theorem we introduce some notation. For any redistribution random variable $X\in[0,1]$ satisfying the symmetry and non-degeneracy assumptions,  
	let $\hat X$ be the size-biased version of $X$, that is, 
	$$\P(\hat X\le s)=2{\E[X\car_{X\le s}]}\comma\qquad s\in[0,1]\,,$$ 
	and call 
	\begin{equation}\label{eq:def-h-s}
		\mathfrak{h}=\E[-\log\hat X]=\E[-2X\log X]\,,\qquad \mathfrak{s}^2=\var(\log\hat X)\,,\qquad \mathfrak{r}=\frac{\mathfrak{s}}{\mathfrak{h}}\,.
	\end{equation}
 Finally, define 
	\begin{equation}
		t_{\rm ent}=\frac{n\log n}{2 \mathfrak{h}}\,,\qquad t_{\rm w}=\left(1+\mathfrak{r}\right)\frac{n}{2}\sqrt{\frac{\log n}{\mathfrak{h}}}\comma
	\end{equation}
	with the subscript \textquotedblleft${\rm ent}$\textquotedblright\ indicating the dependence on the entropic parameter $\mathfrak{h}$, in analogy with the terminology in \cite{bordenave_random_2018,bordenave_cutoff2019}.	
	 \begin{theorem}\label{theorem}
	Let $X\in[0,1]$ be an arbitrary random variable satisfying the symmetry and non-degeneracy assumptions. For all three models, 
for all $\beta\in\R$, 
	\begin{equation}\label{eq:th}
		\lim_{n\to\infty}\left|W_1(t_{\rm ent}+\beta   t_{\rm w})-2\Phi\left(-\frac{\beta(1+\mathfrak{r})}{\sqrt{1+\mathfrak{r}^2}}\right)\right|=0\,,
	\end{equation}
	where $\Phi(a)=\frac1{\sqrt{2\pi}}\int_{-\infty}^ae^{-z^2/2}\, \dd z$. Moreover, \eqref{eq:th} holds with $W_1(t)$ replaced by $W_1(\d_{x_0},t)$ for any fixed  $x_0$.  
	The same statements hold if $W_1$ is replaced by $\cW_1$.
\end{theorem}

\subsection{Remarks, related work, and open problems}
	It is instructive to  specialize the result in Theorem \ref{theorem} to the benchmark case $X\sim{\rm Beta}(\alpha,\alpha)$ for some $\alpha>0$ (see also Figure \ref{fig:h-s}). In this case
	\begin{equation}
		\mathfrak{h}=\psi(2\alpha+1)-\psi(\alpha+1)\,,\qquad\mathfrak{s}^2=-\psi'(2\alpha+1)+\psi'(\alpha+1)\,,
	\end{equation}
	where $\psi$ is the digamma function.
	In particular, $\mathfrak{h}$ is monotonically increasing in $\a$, with 
	$$\lim_{\alpha\downarrow 0}\mathfrak{h}=0\,,
	\qquad\lim_{\alpha\to\infty}\mathfrak{h}=\log(2)\,,$$
	while $\mathfrak{s}^2$ satisfies
	$$\lim_{\alpha\downarrow 0}\mathfrak{s}^2=0\,,
	\qquad \lim_{\alpha\to\infty}\mathfrak{s}^2=0
	\,.$$
Notice that, taking the limit $\alpha\to\infty$, one recovers the singular case $X\equiv 1/2$.
\begin{figure}[h]\label{fig:h-s}
	\centering
	\includegraphics[width=7cm]{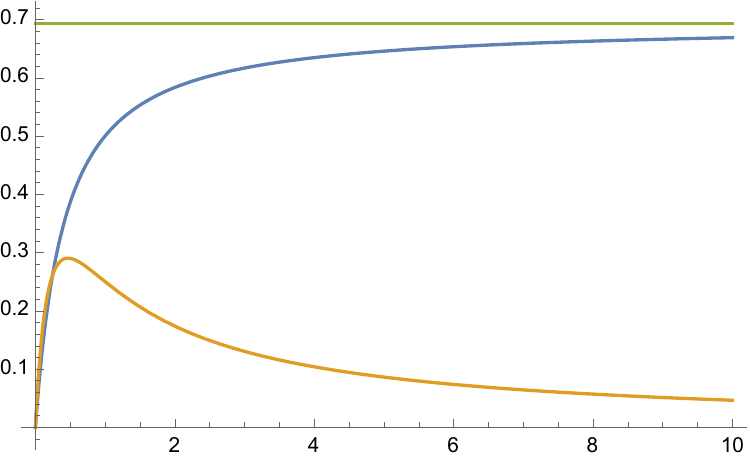}
	\caption{Plot of the functions $\mathfrak{h}$ and $\mathfrak{s}$ given in \eqref{eq:def-h-s} for $X\sim{\rm Beta}(\alpha,\alpha)$, $\alpha>0$. In blue: the function $\alpha\mapsto \mathfrak{h}$. In orange: the function $\alpha\mapsto \mathfrak{s}^2$. In green: the horizontal line at height $\log 2$.}
\end{figure}

As we have seen, when $X\equiv 1/2$, all three models reduce to the mean-field averaging process that was recently analyzed in 
\cite{chatterjee2020phase}. Indeed, our Theorem \ref{theorem} contains the main result of \cite{chatterjee2020phase} as a special case, and may be regarded as a generalization thereof, extending the cutoff phenomenon to a universal behavior in all three families of models with arbitrary redistribution random variable $X$.  The proof of \cite{chatterjee2020phase} is based on a discretization of the energy configuration, while our proof uses a more flexible analysis of a suitable pile dynamics encompassing the energy evolution $\eta_t$; see Section \ref{sec:piledyn}. The analysis of the pile dynamics was recently introduced in  \cite{caputo2024repeated}, where we study a different type of generalization of  the averaging process where equalization occurs for possibly more than two energy variables at a time. We also mention that graphical versions of the averaging  processes have been recently studied, with energy equalizations taking place along the edges of an underlying graph; see \cite{aldous_lecture_2012,quattropani2021mixing,movassagh_repeated2022,caputo_quattropani_sau_cutoff_2023}.

Concerning the SRM, in the special case of Beta redistribution, the mixing time in total variation distance has been studied in  \cite{smith_gibbs_2014}, where it is shown that it is of order $n\log n$, confirming a conjecture in \cite[Section 13.1.4]{aldous-fill-2014}. Establishing the precise location and cutoff  behavior for the total variation distance remains an interesting open problem for these models.  Theorem \ref{theorem} shows that cutoff for the weaker Wasserstein distance occurs on the same timescale. As a consequence, it can be used to obtain meaningful lower bounds on the total variation distance for SRM. In particular, one can check that, in the special case of ${\rm Beta}(\a,\a)$ redistribution variables 	with $\a\in(0,1)$, it provides an estimate which improves the lower bound established 
by a coupon collector argument  in  \cite{smith_gibbs_2014}. 
We remark that similar lower bounds can be deduced, using our results, for the matrix versions of the SRM model considered in  \cite{diaconis_saloff-coste_bounds_2000,oliveira_convergence_2009,jiang_kac_2017}.

\subsection{Organization of the paper}
In Section \ref{sec:piledyn} we introduce the main ideas behind our proofs. In Section \ref{sec:prelim} we collect some preliminary facts as well as the main ingredients to be used in the proof of Theorem \ref{theorem}. Finally, Section \ref{sec:proofs} proves our main results.

	\section{Pile dynamics and proof strategy}
	\label{sec:piledyn}
	\subsection{Pile dynamics}
The proof of Theorem \ref{theorem} will be crucially based on the analysis of a richer representation of the stochastic exchange models introduced above. The associated stochastic process is a version of the splitting process we recently introduced  
 in  \cite{caputo2024repeated}, under the name of \emph{pile dynamics}. Loosely speaking, while the original process is characterized by the variables $\eta_t(x)$, representing  the energy of the $x$-th particle at time $t$, with the initial distribution $\eta_0$ given by the Dirac mass $\d_{x_0}$ at some  $x_0\in V$, the new process shall keep track of the distinct energy fragments produced by each exchange event, so that the energy $\eta_t(x)$ is given by the sum of the energy fragments present at $x$ at time $t$.

In the current setting, the pile dynamics can be formally defined as follows. 	 
 At time zero there is only one pile sitting on particle $x_0$, with energy equal to $\eta_0(x_0)=1$. Whenever there is an event at the pair $(x,y)$, with redistribution variable $X$, each pile sitting at either  $x$ or  $y$ is split  into two new labeled piles according to the following rule, depending on the model under consideration: 

\smallskip

\emph{SRM}: a pile of size $p$ at $x$ is split into two piles of size $Xp$ and $(1-X)p$, respectively, the first stays at $x$ while the second is moved to  $y$;  a pile of size $p$ at $y$ is split into two piles of size $Xp$ and $(1-X)p$, respectively, the first is moved to  $x$, while the second stays at $y$.  

\emph{SEM}: a pile of size $p$ at $x$ is replaced by two piles, both of size $Xp$, the first stays at $x$ while the second is moved to  $y$; a pile of size $p$ at $y$ is replaced by two piles, both of size $(1-X)p$, the first is moved to $x$ while the second stays at $y$.  

\emph{GAM}: a pile of size $p$ at $x$ is split into two piles of size $Xp$ and $(1-X)p$, respectively, the first stays at $x$ while the second is moved to  $y$; a pile of size $p$ at $y$ is split into two piles of size $Xp$ and $(1-X)p$, respectively, the first stays at $y$, while the second is moved to $x$.

\smallskip

In this way, for all three models, at each time $t$, there is at most a finite number of piles at each $x$, and their total size sums up to $\eta_t(x)$. A generic pile is denoted by $\zeta$ and we write $|\zeta|$ for its size. Hence, if $A_t(x)$ denotes the collection of piles at time $t$ at  $x$,  
\begin{equation}\label{eq:piledyn1}
		\eta_t(x)=\sum_{\zeta\in A_t(x)}|\zeta|\,.
	\end{equation}
An important feature of this construction is that the number of elements in the set $A_t(x)$ depends only on the pairs $(x,y)$ chosen at each update. In particular, it does not depend on the realization of the redistribution variables $X$.	

\subsection{Energy from large piles} It will be key to consider truncations of the pile dynamics obtained by restricting to sizes above  some prescribed threshold $\theta\in[0,1]$. With this in mind, we define
	\begin{equation}
		\eta^\theta_t(x)=\sum_{\zeta\in A_t(x)}|\zeta|\car_{|\zeta|\ge \theta}\,.
	\end{equation}
Clearly, $\eta_t^{\theta}=\eta_t$ for $\theta=0$. The following property will be at the heart of our analysis.
	\begin{proposition}[Expected energy from piles above threshold $\theta$]\label{lemma:size-pile}
		For all $t\ge 1$ and $\theta\in [0,1]$,  and for any of the three models, with initial configuration given by a Dirac mass,
		\begin{equation}\label{eq:size-pile}
			\E\big[\|\eta^\theta_t\|_1\big]= \P\bigg(\textstyle{\sum_{i=1}^T}  \log \hat X_i \ge \log \theta\bigg)\comma 
		\end{equation}
		where $\hat X_i$ are i.i.d.\ copies of the size-biased redistribution variable $\hat X$, and $T\sim {\rm Bin}(t,\frac{2}{n})$ is an independent binomial random variable with parameters $t$ and $2/n$.   In particular, for $\theta=0$ (with the convention $\log 0\eqdef -\infty$) we have $\E[\|\eta_t\|_1]=1$ for all $t\ge 1$.
	\end{proposition}
\begin{proof}	
	Let $A_t=\cup_x\, A_t(x)$ denote the set of all piles at time $t$. By keeping track of the number of updates experienced by each individual pile, we may define the variable $u_t(\zeta)$, for any $\zeta\in A_t$, as the number of updates involving the given pile $\zeta$ up to time $t$.  Let $A_{s,t}$ denote the set of piles $\zeta\in A_t$ such that $u_t(\zeta)=s$, $s=0,\dots,t$. Since the variables $u_t$ only depend on the choice of edges $(x,y)$ at each update, the cardinality $|A_{s,t}|$ of the set $A_{s,t}$ is independent of the redistribution random variables $X_i$. Moreover, conditionally on the cardinality of $A_{s,t}$, using the symmetry ${\rm Law}(X)={\rm Law}(1-X)$, and the independence of the redistribution variables, it follows that for each pile  $\zeta\in A_{s,t}$, the size $|\zeta|$ has the law of $ \prod_{i=1}^{s}X_i$, where $X_i$ are i.i.d.\ copies of $X$. 
		Therefore, in all three models, for any $t\in\N$, $\theta\in [0,1]$, 
		\begin{align*}
			\E\left[\|\eta^\theta_t\|_1\right]&= \sum_{s=0}^t\E\bigg[\textstyle{\sum_{\zeta\in A_{s,t}}} |\zeta|\car_{|\zeta|\ge\theta}  \bigg]\\
			&=\sum_{s=0}^t\E\bigg[\E\bigg[\textstyle{\sum_{\zeta\in A_{s,t}}} |\zeta|\car_{|\zeta|\ge\theta} \,\Big\rvert\, |A_{s,t}|\bigg] \bigg]\\
			&=\sum_{s=0}^t\E\big[| A_{s,t} |\big] \E\left[\textstyle{\prod_{i=1}^{s}}X_i\car_{\prod_{i=1}^{s}X_i\ge\theta}\right] \,.
		\end{align*}
	To compute $\E\big[| A_{s,t} |\big]$, let $T={\rm Bin}(t,\tfrac2n)$ denote the binomial random variable with parameters $t$, $2/n$, and observe that 
		\begin{equation}\label{eq:claim-Ast}
			\E\big[| A_{s,t} |\big]=2^s\P\left(T=s\right)\,.
		\end{equation}
		To see this, one may argue as follows. Since $|A_{s,t}|$ is independent of the random variables $X_i$, we may consider the standard averaging process where $X\equiv1/2$. For integers $0\le s\le t$, let $e_{s,t}$ denote the total energy coming from piles of size exactly equal to $2^{-s}$. Then $e_{s,t}=2^{-s}|A_{s,t}|$.  Indeed, $e_{s,t}$ can be obtained by summing over all piles which experienced exactly $s$ updates. Moreover, at any given time, the probability that a specified pile experiences an update is equal to $2/n$.  In conclusion, 
		\begin{equation}
			2^{-s}\E\big[| A_{s,t} |\big]=\E[e_{s,t}]=\P(T=s)\,, 
		\end{equation}
		which proves \eqref{eq:claim-Ast}.
				Once we have the identity \eqref{eq:claim-Ast}, it follows that 
			\begin{align*}
			\E\left[\|\eta_t^\theta\|_1\right ]
			&=\sum_{s=0}^t \P\left(T=s\right) \E\left[\textstyle{\prod_{i=1}^{s}}2X_i\car_{\prod_{i=1}^{s}X_i\ge\theta}\right] \\
&=\sum_{s=0}^t \P\left(T=s\right)\P\left(\textstyle{\prod_{i=1}^{s}} \hat X_i\ge\theta\right)=\P\left(\textstyle{\prod_{i=1}^{T}}\hat X_i\ge \theta\right)\,.
		\end{align*}
		The proof is concluded by taking logarithms.
	\end{proof}

	\subsection{Strategy of proof}
	Let us briefly outline the main steps in the proof of Theorem \ref{theorem}.
	Once we have the identity in Proposition \ref{lemma:size-pile}, the proof   of Theorem \ref{theorem} will be based on the following steps. 
	 	 Define $\psi =\gamma \sqrt{\mathfrak h\log n}$, with a parameter $\gamma>0$ to be tuned, and write $\eta^{+,\g}_t:=\eta^\theta_t$ with $\theta:=\frac1ne^\psi$, that is,
\begin{equation}\label{eq:def-eta+g}
\eta_t^{+,\g}(x)=\sum_{\zeta\in A_t(x)} |\zeta|\car_{n|\zeta|\ge e^{\psi}}\,.
\end{equation}
Recall the notation from Theorem \ref{theorem}. 
We will see that, by
an application of the central limit theorem, Proposition \ref{lemma:size-pile} implies the following statement. 
\begin{corollary}\label{coro:CLT-pile}
		Fix $\gamma>0$ and let $\psi=\gamma\sqrt{\mathfrak{h}\log n}$, $t=t_{\rm ent}+\beta t_{\rm w}$ for some $\beta\in\R$ and define $\eta^{+,\g}_t$ as in \eqref{eq:def-eta+g}, with  $\eta_0=\d_{x_0}$. Then
		\begin{equation}
			\lim_{n\to \infty}	\left|\E\left[\|\eta^{+,\g}_t\|_1\right]-\Phi\left(-\frac{\beta(1+\mathfrak{r})+\gamma}{\sqrt{1+\mathfrak{r}^2}}\right)\right|=0 \,.
		\end{equation}
	\end{corollary}
The proof will be given in Section \ref{sec:prelim}.
Thanks to Corollary \ref{coro:CLT-pile},
the proof of the lower bound in Theorem \ref{theorem} can be reduced to proving a lower bound of the form 
\begin{equation}\label{eq:eta+2}
\liminf_{n\to\infty} \left(W_1(t)-  2\E\left[\|\eta^{+,\g}_t\|_1\right]\right)\ge 0\,,
\end{equation}
for all $\g>0$, where $\eta^{+,\g}_t$ is taken as in Corollary \ref{coro:CLT-pile}. Taking the limit $\g\to 0$ then establishes the desired lower bound. 
We shall actually establish \eqref{eq:eta+2} for the distance $\cW_1(t)$; see Section \ref{sec:proofs}. On the other hand,  the upper bound in Theorem \ref{theorem} will be reduced to proving a bound of the form 
\begin{equation}\label{eq:eta+22}
\limsup_{n\to\infty} \left(W_1(t_\g)-   2\E\left[\|\eta^{+,\g}_{t_\g}\|_1\right]\right)\le 0\,,
\end{equation}
where $t_\g = t_{\rm ent}+\beta t_{\rm w} + c\g n\sqrt{\log n}$, for a suitable constant $c>0$, and $\eta^{+,\g}_t$ is taken as in Corollary \ref{coro:CLT-pile}.  The conclusion then follows by using Corollary \ref{coro:CLT-pile} and then taking the limit $\g\to 0$. 

The proof of \eqref{eq:eta+2}--\eqref{eq:eta+22} will use, among other things,  the fact that once we have removed the piles with height larger than $\theta=\frac1ne^\psi$, an $L^2$ estimate will be sufficient to control the distance to stationarity. In turn, the analysis of the $L^2$ norm will be based on the following estimates which take a slightly different form depending on the model.

In analogy with \eqref{def:W1-dritto}, let $W_2$ denote the $2$-Wasserstein distance
\begin{align}\label{def:W2-dritto}
	W^2_2(\eta_0,t)	
	= \inf_{(\xi,\vartheta)\sim(\eta_t,\eta_{\infty})}\E\big[\|\xi-\vartheta\|^2_2\big]\,,
\end{align}
where $\|v\|_2$ denotes the $L^2$-norm of the vector $v$:  $\|v\|^2_2=\sum_{x=1}^nv(x)^2$. 
\begin{lemma}[$W_2$ contraction for SRM and GAM]\label{lemma:AL}
For any $\eta_0\in \D$, $t\in\N$, the SRM and GAM satisfy
	\begin{equation}\label{eq:2contraction}
		W_2^2(\eta_0,t)\le 	e^{-\lambda t}W_2^2(\eta_0,0)\,,
	\end{equation}
	where the rate $\lambda$ is given by, respectively in the two models,
\begin{equation}\label{eq:2gap}
		\lambda_{\rm SRM} = 
		\frac2n \tonde{1-2\E[X^2]\frac{n-2}{n-1}}\,,\quad 
		\lambda_{\rm GAM} = \frac{2}{n-1}(1-2\E[X^2])\,.
	\end{equation}
	\end{lemma}
In the SEM, because of the non-conservation of total energy, we shall instead use the following bound to control the $L^2$-norm. 
We write $\langle \eta\rangle = \frac1n\sum_x \eta(x)$ for the average energy per particle of the configuration $\eta \in \R_+^n$. 
\begin{lemma}[$W_2$ contraction for SEM]\label{lemma:AL_HMP}
For any $\eta_0\in \D$ and $t\in\N$, the SEM satisfies
	\begin{equation}\label{eq:contraction-hpm}
	\E	\left[\|\eta_t-\langle \eta_t\rangle\|_2^2\right]\le 
		e^{-\lambda t} 	\|\eta_0-\langle \eta_0\rangle\|_2^2,
	\end{equation}
	with rate $\lambda$ given by
\begin{equation}\label{eq:2gaps}
		\lambda_{\rm SEM} = 
		\frac{2}{n-1}(1-2\E[X^2])\tonde{1+ \frac{4\E[X^2]-1}{n(1-2\E[X^2])}}\,.
	\end{equation}
	\end{lemma}
	We remark that all rates $\l$ appearing in Lemma \ref{lemma:AL} and Lemma \ref{lemma:AL_HMP} have the same asymptotic behavior
	 	\begin{equation}\label{eq:limit-gap}
		\lambda\sim \frac2n \left(1-2\E[X^2]\right)\,.
	\end{equation}
When $X\equiv1/2$, all three models reduce to the standard averaging process, in which case $\l=1/(n-1)$ and the above estimates are well known; see, e.g., \cite{aldous_lecture_2012,chatterjee2020phase}.

\section{Preliminary facts and main tools}\label{sec:prelim}

\subsection{Coupling, duality and stationary measures}
We may couple the three models in such a way that the updated pairs $(x,y)$ and the redistribution variables $X$ coincide at each step.  
Recalling the representation \eqref{eq:repre}-\eqref{eq:repre1}, and writing $\eta,\xi,\omega$ respectively for the SRM, SEM, and GAM configurations, we have
\begin{equation}\label{eq:repres}
	\eta_t = \eta_0 R_1\cdots R_t\comma\qquad \xi_t =  \xi_0 R^{\top}_1\cdots R^{\top}_t\comma\qquad 
	 \omega_t =\omega_0 Q_1\cdots Q_t\comma
\end{equation}
where $R^{\top}_i$ denotes the transpose of $R_i$, and $R_i$ and $Q_i$ denote the $n\times n$ matrix given by the identity except for the $2\times 2$ block corresponding to the $(x,y)$ pair, which takes the form $r$, $q$ respectively; see \eqref{eq:repre1}. The representation \eqref{eq:repres} immediately implies the following distributional identity, valid for any initial $\xi_0,\eta_0\in\D$,
\begin{equation}\label{eq:dual}
	\scalar{\eta_0}{\xi_t}\overset{d}{=}\scalar{\eta_t}{\xi_0} ,
\end{equation}
where $\scalar{u}{v}=\sum_xu(x)v(x)$ denotes the usual scalar product. This is an instance of the duality relations between SRM and SEM recently discussed in \cite{de_masi_ferrari_gabrielli_hidden_2023,giardina_redig_tol_intertwining_2024,kim2025spectral}.
We turn to a discussion of the stationary distributions and the proof of Proposition \ref{prop:uniqueness} and Lemma \ref{lem:gamma}. 

\subsubsection{Stationary distribution for SRM}
We have $\eta_t\in\D$, for all $t\ge 0$. Thus,  existence of a stationary distribution follows from a standard compactness argument. To prove uniqueness, we use the following contraction property. 
\begin{lemma}\label{lem:contra1}
	For any vector $v\in\R^n$, 
	\begin{equation}
		\E\quadre{\|vR_1\|_2^2}=	(1-\lambda_{\rm SRM})\|v\|_2^2+\frac{4\E[X^2]n}{n-1}\langle v\rangle^2\,, 
	\end{equation} 
	where $\lambda_{\rm SRM}$ is given in \eqref{eq:2gap}. 
\end{lemma}
\begin{proof}
By definition of the random matrix  $R_1$,
\begin{align*}
\E\quadre{\|vR_1\|_2^2}&=	\|v\|_2^2 +\frac1{\binom{n}{2}}\sum_{x<y}\left[(2\E[X^2]-1)(v(x)^2+v(y)^2) + 4\E[X^2]v(x)v(y)\right]
\\
& = \|v\|_2^2 \Big(1+\frac2n(2\E[X^2]-1) - \frac{2\E[X^2]}{\binom{n}{2}}\Big) + \frac{4n\E[X^2]}{n-1}\langle v\rangle^2\,.
\end{align*}
Rearranging terms yields the claimed identity. 
\end{proof}
Using linearity, and the fact $ \langle\eta_t-\eta_t' \rangle = 0$ for all $t\ge 0$, iteration of Lemma \ref{lem:contra1} shows that, for any two initial configurations $\eta_0,\eta_0'\in\D$, the evolutions $\eta_t,\eta_t'$ can be coupled in such a way that 
\begin{equation}\label{eq:kmpco}
\E\quadre{\|\eta_t-\eta_t'\|_2^2}	=(1-\lambda_{\rm SRM})^t\|\eta_0-\eta_0'\|_2^2\,, \qquad t\geq 0\,.
\end{equation}
Since $ \lambda_{\rm SRM}\ge \frac{2}{n(n-1)}>0$,   $\E\quadre{\|\eta_t-\eta_t'\|_2^2}
\to 0$ as $t\to\infty$ uniformly in the initial conditions $\eta_0,\eta_0'\in\D$. It follows that there exists a unique stationary distribution $\r$ on $\D$ for SRM, and for any initial distribution $\nu$ on $\D$, if $\nu_t$ denotes the law at time $t$, then $\nu_t$ converges weakly to $\r$ as $t\to\infty$.

\subsubsection{Stationary distribution for GAM} 
The analogue of Lemma \ref{lem:contra1} here is as follows. 
\begin{lemma}\label{lem:contra2}
	For any vector $v\in\R^n$, 
	\begin{equation}
		\E\quadre{\|vQ_1\|_2^2}=	(1-\lambda_{\rm GAM})\|v\|_2^2+n\lambda_{\rm GAM}  \langle v\rangle^2\,, %\quad 
			\end{equation} 
			where 
		$\lambda_{\rm GAM}$ is given in \eqref{eq:2gap}.
\end{lemma}
\begin{proof}
By definition of the random matrix  $Q_1$,
\begin{align*}
\E\quadre{\|vQ_1\|_2^2}&=	\|v\|_2^2 +\frac1{\binom{n}{2}}\sum_{x<y}\left[(2\E[X^2]-1)(v(x)^2+v(y)^2) + 4\E[X(1-X)]v(x)v(y)\right]
\\
& = \|v\|_2^2 \Big(1+\frac2n(2\E[X^2]-1) - \frac{2\E[X(1-X)]}{\binom{n}{2}}\Big) + \frac{4n\E[X(1-X)]}{n-1}\langle v\rangle^2\,.
\end{align*}
Using $2\E[X(1-X)]=1-2\E[X^2]$, and rearranging, one has the desired identity. 
\end{proof}
We have $\omega_t\in\D$, for all $t\ge 0$. Moreover, it is immediate to check that the Dirac mass at the constant vector $\bar u(x)\equiv \frac1n$ is a stationary distribution for the GAM. To prove uniqueness and convergence observe that by taking $v=\omega_0-\bar u$ in Lemma \ref{lem:contra2} and iterating one obtains
 \begin{equation}\label{eq:kmpco2}
\E\quadre{\|\omega_t-\bar u\|_2^2}	=(1-\lambda_{\rm GAM})^t\|\omega_0-\bar u\|_2^2\,, \qquad t\geq 0\,.
\end{equation}
Since $\lambda_{\rm GAM}>0$ by the non-degeneracy assumption on $X$, this proves that for any $\omega_0\in\D$, the law of $\omega_t$ converges weakly to the Dirac mass at $\bar u$. In particular, it establishes uniqueness of the stationary distribution. 

\subsubsection{Proof of Proposition \ref{prop:uniqueness}}
The above discussion proves the statements concerning the SRM and GAM. In order to prove the statements about SEM, first notice that
by the maximum principle \eqref{eq:max} we may restrict the state space to $\xi\in\R_+^n$ such that $\max_{x}\xi_{0}(x)\le 1$. 
Then, by compactness, there exists stationary distributions for the SEM. In contrast with the other two processes, here we do not have uniqueness. This can be seen already from \eqref{eq:dual}: if e.g.\ $\eta_0(x)\equiv 1/n$, then the law of $ \langle \xi_t\rangle = \scalar{\eta_0}{\xi_t}$
converges to the law of the random variable $Y:=\scalar{\eta_\infty}{\xi_0}$, which depends on the initial condition $\xi_0$.  Here $\eta_\infty$ stands for the random variable with distribution $\r$, where $\r$ is the unique stationary distribution for the SRM. It remains to show that $\xi_t$ is asymptotically flat. 
To see this,  we note that by symmetry  of $X$, as in Lemma \ref{lem:contra2} 
one has, for any vector $v \in\R^n$,
\begin{align*}
\E\quadre{\|vR_1^\top\|_2^2}=\E\quadre{\|vQ_1\|_2^2}=
(1-\lambda_{\rm GAM})\|v\|_2^2+n\lambda_{\rm GAM}  \langle v\rangle^2
\end{align*}
Therefore, for any $\xi\in\R_+^n$, writing $\xi'=\xi R_1^\top$,
\begin{align*}
\E\quadre{\|\xi' - \langle \xi' \rangle \|_2^2}&=\E\quadre{\|\xi'\|_2^2 }- n\E\quadre{\langle \xi'\rangle^2}\\
 &  = 
(1-\lambda_{\rm GAM})\|\xi\|_2^2+n\lambda_{\rm GAM}  \langle \xi\rangle^2- n\E\quadre{\langle \xi'\rangle^2}
\,.
\end{align*}
A similar computation  shows that 
\begin{align}\label{eq:sipo}
\E\quadre{\langle \xi'\rangle^2}
= \left(1-\frac{2(4\E[X^2]-1)}{n(n-1)}\right)\langle \xi\rangle^2 +\frac{2(4\E[X^2]-1)}{n^2(n-1)}\,\|\xi\|_2^2\,.
\end{align}
In conclusion,
\begin{align}\label{eq:SEMt}
\E\quadre{\|\xi' - \langle \xi' \rangle \|_2^2}&=\left(1-\lambda_{\rm GAM}-\frac{2(4\E[X^2]-1)}{n(n-1)}\right)
\left(\|\xi\|_2^2 - n\langle \xi\rangle^2\right)\nonumber\\
&=(1-\lambda_{\rm SEM})\|\xi - \langle \xi \rangle \|_2^2\,,
\end{align}
where $\lambda_{\rm SEM} = \lambda_{\rm GAM}+\frac{2(4\E[X^2]-1)}{n(n-1)}$ as in \eqref{eq:2gaps}.
Iterating the last identity proves the exponential decay stated in Lemma \ref{lemma:AL_HMP}. In particular, for any $\xi_0\in\D$, $\xi_t$ converges weakly as $t\to\infty$ to the constant configuration with random height $Y=\scalar{\eta_\infty}{\xi_0}$. This concludes the proof of Proposition \ref{prop:uniqueness}.

\subsubsection{Proof of Lemma \ref{lem:gamma}}
In the special case where $X\sim{\rm Beta}(\alpha,\alpha)$ for some $\a>0$, the unique stationary distribution $\r$ of SRM is the  Dirichlet distribution ${\rm Dir}(\a)$ obtained as the product of i.i.d.\  ${\rm Gamma}(\alpha,1)$ random variables conditioned to have total sum equal to $1$. Indeed, it is easily checked that in this case an update at $(x,y)$ coincides with the ``heat bath'' update with respect to ${\rm Dir}(\a)$, obtained by replacing the pair $(\eta_x,\eta_y)$ by a random pair $(\eta'_x,\eta'_y)$ distributed according to the conditional distribution ${\rm Dir}(\a)$ given the configuration $(\eta_z, z\neq x,y)$. This also shows that the SRM is a reversible Markov chain. For a partial converse implication, see, e.g., \cite{grigo_mixing_2012}.

We turn to the statement concerning the SEM. The marginal of ${\rm Dir}(\a)$ on a single component is given by $Y_\a:={\rm Beta}(\alpha,\alpha(n-1))$, and thus the duality \eqref{eq:dual} implies that $\langle \xi_t \rangle$ converges weakly to $Y_\a$ when the initial configuration $\xi_0=\d_{x_0}$ is a Dirac mass at  $x_0$. Together with the arguments in the proof of  Proposition \ref{prop:uniqueness} this ends the proof. 

\subsubsection{Proof of Lemma \ref{lemma:AL} and Lemma \ref{lemma:AL_HMP}}\label{sec:proof-L2-contraction}
 Let $\eta_\infty$ denote the random variable with law $\r$, where $\r$ is the unique stationary distribution of SRM. Then writing $R(t)=R_1\cdots R_t$, we have that $(\eta_0R(t),\eta_\infty R(t))$ is a coupling of $(\eta_t,\eta_\infty)$. Therefore, 
\begin{align}\label{W2-con}
	W^2_2(\eta_0,t)\le	
	\E\big[\| \eta_0R(t) - \eta_\infty R(t)\|^2_2\big]\,.
\end{align}
By using \eqref{eq:kmpco}, we obtain, for any fixed $\eta_0\in\D$,  
\[W^2_2(\eta_0,t)\le e^{-\l_{\rm SRM}t}\E\big[\| \eta_0 - \eta_\infty \|^2_2 = e^{-\l_{\rm SRM}t} W^2_2(\eta_0,0)\,.
\]
The same argument works for GAM, provided we use \eqref{eq:kmpco2} instead of  \eqref{eq:kmpco}. This ends the proof of Lemma \ref{lemma:AL}.
Finally, we observe that Lemma \ref{lemma:AL_HMP} is obtained by iteration of \eqref{eq:SEMt}.

\subsubsection{Proof of Lemma \ref{lemma:monotone}}
To prove monotonicity of the $W_1$ distance in the SRM model we argue as follows. Let $u,v\in \D$ denote two initial configurations, and write $u',v'$ for their evolution after one step given the pair $(x,y)$ to be updated and the redistribution variable $X$. Then 
\begin{align*}
&\|u'-v'\|_1 -\|u-v\|_1  = |X(u(x)+u(y))- X(v(x)+v(y))| +\\
&\; +  |(1-X)(u(x)+u(y))- (1-X)(v(x)+v(y))| - |u(x)-v(x)|- |u(y)-v(y)| \le 0\,.
\end{align*}
Take $u=\eta_t,v=\eta_\infty$ and let them be coupled by the optimal coupling for the $L^1$ distance $W_1$ at time $t$. Clearly,  $\eta_tR_{t+1}, \eta_\infty R_{t+1}$ provides a coupling for $\eta_{t+1},\eta_\infty$, and the above inequality ensures that  $ \|\eta_tR_{t+1}-\eta_\infty R_{t+1}\|_1\le \|\eta_{t}-\eta_\infty\|_1$. Therefore, $W_1(\eta_0,t+1)\le  W_1(\eta_0,t)$ for all $\eta_0\in\D$. The same argument with minor modifications works also for SEM and GAM. Moreover, the same argument can be used to prove the monotonicity for the weaker distance $\cW_1(t)$. The only difference is that the above computation has to be performed by using the pair $(\sigma(x),\sigma(y))$ for the variable $v'$ when the pair $(x,y)$ is chosen for the variable $u'$, where  $\sigma\in \mathfrak S_n$ is the optimal permutation in $\inf_{\sigma\in\mathfrak{S}_n} \|u-v\circ \sigma\|_1$. The mean-field nature of the process allows one for such a modification and the monotonicity is preserved.

\subsection{Proof of Corollary \ref{coro:CLT-pile}}
			By our choice of $t$ and the fact that $T\sim{\rm Bin}(t,\frac2n)$,  as $n\to\infty$ one has
		\begin{equation}
		\label{eq:poire}
			\var(T)\sim \E[T]\sim \frac{\log n}{ \mathfrak{h}}\,.
		\end{equation}
		Let us fix some $\alpha\in\R$ and define 
		\begin{equation}
				m:=m(\alpha)=\lfloor \E[T]+\alpha\sqrt{\var T} \rfloor\,.
		\end{equation}
		Then, 
			\begin{align}
			m&=\frac2n t+\alpha\sqrt{\frac2n t}+O(1)\\
			&=\frac{\log n}{\mathfrak h}+\beta\left(1+\mathfrak r\right)\sqrt{\frac{\log n}{\mathfrak h}}+\alpha\sqrt{\frac{\log n}{\mathfrak h}}+o\left(\sqrt{\log n}\right)\,.
		\end{align}	 
		Therefore, using $\theta=n^{-1}e^\psi$, \begin{equation}
			\P\left(\sum_{s=1}^m\log(\hat X_s)\ge\log \theta\right)=\P\left(\frac1{\sqrt m}\sum_{s=1}^m \frac{\log(\hat X_s^{-1})-\mathfrak h}{\mathfrak s}\le \frac{\log n -\psi -m \mathfrak h}{\mathfrak s \sqrt{m}}\right)\,.
		\end{equation}
		Notice that
		\begin{equation}
			\log n -\psi -m \mathfrak h=\sqrt{\mathfrak h\log n}\big(-\alpha-\gamma-\beta(1+\mathfrak r) \big)+o(\sqrt{\log n})\,.
		\end{equation}
		Hence, 
			\begin{equation}
			\frac{\log n -\psi -m \mathfrak h}{\mathfrak s \sqrt m}=\frac{1}{\mathfrak r}\big(-\alpha-\gamma-\beta(1+\mathfrak r) \big)+o(1)\,.
		\end{equation}
	Taking the limit $n\to\infty$, and using the CLT, 
		\begin{equation}
				\lim_{n\to\infty}\P\left(\sum_{s=1}^m\log(\hat X_s)\ge\log \theta\right)=\Phi\left(\frac{1}{\mathfrak r}\big(-\alpha-\gamma-\beta(1+\mathfrak r) \big) \right)\,.
		\end{equation}
		By Proposition \ref{lemma:size-pile},
		\begin{equation}
			\begin{split}
			\E[\|\eta^\theta_t\|_1]&=\sum_{r=0}^{t}\P(T=r)	\P\left(\sum_{s=1}^r\log(\hat X_s)\ge\log \theta\right)\\
			&=\int_{\R}\P(T=m(\alpha))\sqrt{\var(T)}\,\P\left(\sum_{s=1}^{m(\alpha)}\log(\hat X_s)\ge\log \theta\right){\rm d}\alpha\,.
			\end{split}
		\end{equation}
	The local CLT for binomial random variables  implies, for all $\alpha \in \R$,
		\begin{equation}\label{eq:lclt}
			\lim_{n\to\infty}\left|\P(T=m(\alpha))\sqrt{\var(T)}-\varphi(\alpha)\right|=0\,,
		\end{equation}
		where $\varphi$ is the density of a standard Gaussian. 
			Letting 
			\begin{equation}
				\varphi_n(\alpha)=\P(T=m(\alpha))\sqrt{\var(T)}\,,\qquad f_n(\alpha)=\P\left(\sum_{s=1}^{m(\alpha)}\log(\hat X_s^{-1})<\log n-\psi \right)\,,
			\end{equation}
we rewrite
		\begin{equation}
			\E[\|\eta^\theta_t\|_1]
			=\int_{\R}\varphi_n(\alpha)f_n(\alpha)\,{\rm d}\alpha\, ,
	\end{equation}
	and claim that this expression converges, as $n\to \infty$, to
	\begin{equation}
		\int_\R \varphi(\alpha)\, f(\alpha)\, \dd \alpha = \Phi\left(-\frac{\beta(1+\mathfrak{r})+\gamma}{\sqrt{1+\mathfrak{r}^2}}\right) \comma\qquad \text{with}\ f(\alpha)= \Phi\bigg(\frac{1}{\mathfrak r}\big(-\alpha-\gamma-\beta(1+\mathfrak r) \big) \bigg) \fstop
	\end{equation}
	Indeed, we have
	\begin{align}
		\left|\int \varphi_n\, f_n - \int \varphi\, f \right|&\le \int \abs{\varphi_n-\varphi}f_n + \int \varphi\abs{f_n-f}\le \int \abs{\varphi_n-\varphi} + \int \varphi \abs{f_n-f}\comma
	\end{align}
	where the last step used that $f_n\le 1$.
		Thanks to \eqref{eq:lclt} and   $\int\varphi_n= \int\varphi$, the first integral on the right-hand side vanishes as $n\to \infty$ by Scheff\'{e}'s lemma. Since, by the CLT, we have $\lim_{n\to \infty}f_n=f$ pointwise,  the second integral goes to $n\to \infty$ by the dominated convergence theorem. This concludes the proof of the corollary.
	
	\section{Proof of Theorem \ref{theorem}}\label{sec:proofs}
	\subsection{The lower bound}
We prove the lower bound for $\cW_1(t)$, since  $W_1(t)\ge \cW_1(t)$ for all $t\ge 0$.
We give a unified argument for all three models, and denote $\eta_t$ the configuration at time $t$.  
Let $(\eta_t,\eta_\infty)$ denote an arbitrary coupling of the random variables $\eta_t$ and $\eta_\infty$ corresponding to the initial condition $\eta_0 = \d_{x_0}$ given by a Dirac mass. 
We fix $\g>0$ and $\b\in \R$, set $\psi =\gamma \sqrt{\mathfrak h\log n}$ and $t=t_{\rm ent}+\beta t_{\rm w}$, and omit the parameter $\g>0$ from our notation. Henceforth, we simply write $\eta_t^+=\eta_t^{+,\gamma}$  (see \eqref{eq:def-eta+g}), and define $\eta_t^-=\eta_t-\eta_t^+$.

We are going to show that
\begin{equation}\label{eq:eta+a2}
\liminf_{n\to\infty} 
\left(\cW_1(t)-   2\E\left[\|\eta^{+}_t\|_1\right]\right)\ge0\,.
\end{equation}
As discussed in \eqref{eq:eta+2}, this is sufficient to prove the desired lower bound.  
To prove \eqref{eq:eta+a2}, we write, for all $\sigma \in \mathfrak S_n$,
\begin{align*}
	&\|\eta_t-\eta_\infty\circ\sigma\|_1 = \textstyle\sum_{x} |\eta_t(x)-\eta_\infty(\sigma(x))|\\
	&=\textstyle\sum_{x}\abs{\eta_t(x)-\eta_\infty(\sigma(x))}\car_{\eta_t^+(x)> 0} + \textstyle\sum_{x}\abs{\eta_\infty(\sigma(x))-\eta_t^-(x)}\car_{\eta_t^+(x)=0}\\
	&\ge \|\eta_t^+\|_1
	-\textstyle\sum_{x} \eta_\infty(\sigma(x))\car_{\eta_t^+(x)>0} + \textstyle\sum_{x}\eta_\infty(\sigma(x))\car_{\eta_t^+(x)=0}-\sum_{x\in V}\eta_t^-(x)\car_{\eta_t^+(x)=0}\\
	&\ge \|\eta_t^+\|_1 +\|\eta_\infty\|_1-\|\eta_t^-\|_1-2\textstyle\sum_{x}\eta_\infty(\sigma(x))\car_{\eta_t^+(x)>0}\\
	&= 2\|\eta_t^+\|_1+\|\eta_\infty\|_1-\|\eta_t\|_1 - 2\textstyle\sum_{x}\eta_\infty(\sigma(x))\car_{\eta_t^+(x)>0}\,.
\end{align*}
Next, we estimate 
\begin{align*}
\textstyle\sum_{x}\eta_\infty(\sigma(x))\car_{\eta_t^+(x)>0}\le \|\eta_\infty\|_2\sqrt{\abs{\set{x:\,\eta_t^+(x)>0}}}\le \|\eta_\infty\|_2 \sqrt{\frac{n}{e^\psi}\|\eta_t\|_1}\,.
\end{align*}
We obtained a lower bound independent of  $\sigma\in \mathfrak S_n$. Hence, taking the expectation on both sides and applying the Cauchy-Schwarz inequality,
\begin{align}
	\cW_1(t)&\ge 2 \E[\|\eta_t^+\|_1]+\E[\|\eta_\infty\|_1]-\E[\|\eta_t\|_1]-2\sqrt{ne^{-\psi}\E\quadre{\|\eta_\infty\|_2^2} \E[\|\eta_t\|_1]}\\
	&= 2 \E[\|\eta_t^+\|_1]-2\sqrt{ne^{-\psi}\E\quadre{\|\eta_\infty\|_2^2}} \comma
\end{align}
where we used $\E[\|\eta_t\|_1]=\E[\|\eta_\infty\|_1]=1$. Since $e^{-\psi}\to0$, the conclusion \eqref{eq:eta+a2} follows once we  prove
\begin{equation}\label{eq:L2-norm-eq}
	\sup_{n\in \N} \,n\,\E\quadre{\|\eta_\infty\|_2^2}<\infty\fstop
\end{equation} 
We remark that up to now everything applies equally well to all three models. Moreover, for the GAM we know that $\eta_\infty\equiv 1/n$ and thus \eqref{eq:L2-norm-eq} holds trivially, since $n\|\eta_\infty\|_2^2=1$ in this case. For the SRM we observe that by iterating Lemma \ref{lem:contra1} one finds, for all $t\ge 1$,
\begin{equation}
		\E[\|\eta_t\|_2^2] = \tonde{1-\lambda_{\rm SRM}}^t \|\eta_0\|_2^2 + \frac{4\E[X^2]}{n(n-1)}\sum_{s=0}^{t-1}\tonde{1-\lambda_{\rm SRM}}^s  \,.
	\end{equation} In particular, 
	\begin{equation}\label{eq:UB-kmp-1}
		n\,\E[\|\eta_\infty\|_2^2] = \frac{4\E[X^2]}{\lambda_{\rm SRM}\tonde{n-1}} 
		\sim \frac{2\E[X^2]}{1-2\E[X^2]}\comma
	\end{equation}
	and the right-hand side above is finite by the non-degeneracy assumption on $X$.
Finally, for the SEM, we know that $\eta_\infty$ is a.s.\ constant and thus $\E[\|\eta_\infty\|_2^2]= n \E[\left \langle \eta_\infty\right\rangle^2]$. By the duality \eqref{eq:dual}, it follows that $\E[\left \langle \eta_\infty\right\rangle^2] = \frac1n \E[\|\eta^{\rm SRM}_\infty\|_2^2]$ where $\eta^{\rm SRM}_\infty$ denotes the stationary state of the SRM. In conclusion, $\E[\|\eta_\infty\|_2^2]= \E[\|\eta^{\rm SRM}_\infty\|_2^2]$, and the desired bound \eqref{eq:L2-norm-eq} for the SEM follows using again  \eqref{eq:UB-kmp-1}. 
This ends the proof of the lower bound. 
	\subsection{The upper bound}
	We prove the upper bound for the distance $W_1$. Since $W_1(t)\ge \cW_1(t)$, for all $t\ge 0$, this implies the result for $\cW_1$ as well. 	We now focus on proving the desired estimate \eqref{eq:eta+22} for the SRM first and then discuss the minor modifications in the argument for the other two models. For any coupling of $ (\eta_{t},\eta_{\infty})$ one has 
	\begin{align}\label{eq:three0}
			W_1(\eta_0,t)\le  \E\left[\left\|\eta_{t}-\eta_{\infty} \right\|_1 \right] \,.
	\end{align}
Using linearity we have $W_1(\eta_0,t)\le K(t)$ for any $t\ge 0$ and $\eta_0\in\D$, with $K(t)$ defined by  
	\begin{align}\label{eq:three01}
			K(t):=  \E\left[\left\|\eta_{t}-\eta_{\infty} \right\|_1 \right] \,, \qquad \eta_0=\d_{x_0}\,,
	\end{align}
	where $x_0$ is any fixed particle, and ($ \eta_{t},\eta_{\infty}$) are coupled through the coupling inherited from \eqref{eq:repres}. Thus the desired upper bound  will follow if we prove \eqref{eq:eta+22} with $W_1(t_\g)$ replaced by $K(t_\g)$.
Since now we are taking the initial configuration given by a Dirac mass, this will establish the sought upper bound on both $W_1(\d_{x_0},t)$ and  $W_1(t)$, thus concluding the proof of Theorem \ref{theorem}.
	
	Set $\gamma>0$, $\beta\in\R$, $\psi=\gamma\sqrt{\mathfrak h\log n}$ and $t=t_{\rm ent}+\beta t_{\rm w}$. Recall the definition \eqref{eq:def-eta+g} of $\eta_{t}^+=\eta_t^{+,\gamma}$, which represents the energies coming from large piles. We shall also need to control the energy coming from small piles $\eta_{t}^-=\eta_{t}-\eta_{t}^+$. To this end, we define
\begin{equation}\label{eq:def-eta-star}
		\eta_{t}^*(x)=\eta_{t}^-(x)\car_{n\eta_{t}^-(x)< e^{3\psi}}\,,\qquad \widehat \eta_t = \eta_{t}-\eta_{t}^{*}\,.
	\end{equation}
	By definition, one has the deterministic bound 
\begin{equation}\label{eq:tous}
\|\eta_{t}^*\|_2^2 \le n^{-1}e^{6\psi}=o(1)\,.
\end{equation}
We note that 
	\[ \widehat \eta_t(x)=\eta_{t}^+(x) + \eta_{t}^-(x)\car_{n\eta_{t}^-(x)\ge e^{3\psi}}\,.
	\]
Roughly speaking, $\widehat \eta_t$	consists of energies that are either coming from large piles or from a large number of small piles.

	Take $r=c\gamma n\sqrt{\log n}$ for a suitable constant $c>0$. We start at $\eta_0\in \D$, 
	and consider a fixed realization of $\eta_t=(\eta_{t}^{*},\widehat \eta_t)$. We use this as the initial configuration for a new evolution of $r$ steps of the SRM. Namely, letting $R_i$ denote the matrices in \eqref{eq:repres}, write $(\eta_{t,r}^*,\widehat\eta_{t,r})$ for the joint evolution obtained as
	\begin{equation}
\eta_{t,r}^*=\eta_{t}^*R_{t+1}\cdots R_{t+r}\,,\qquad \widehat\eta_{t,r}=\widehat\eta_{t}R_{t+1}\cdots R_{t+r}\,.
	\end{equation}
	We may couple these evolutions with $\eta_{t+r}$ in the canonical way, using 
	\begin{equation}\label{eq:GC}
\eta_{t}=\d_{x_0}R_{1}\cdots R_{t}\,,\qquad \eta_{t+r}=\eta_{t}R_{t+1}\cdots R_{t+r}\,.
	\end{equation}
Using \eqref{eq:three01}-\eqref{eq:GC}, 
	\begin{align}\label{eq:three}
K(t+r) & \le 
			\E\left[\left\|\eta_{t+r}-\eta_{t,r}^*\right\|_1\right] + \E\left[\left\|\eta_{t,r}^*-\eta_{t,\infty}^*\right\|_1\right]+ \E\left[\left\|\eta_{t,\infty}^*-\eta_\infty\right\|_1\right]
	\end{align}
Since  $\eta_{t+r}=\eta_{t,r}^*+\widehat\eta_{t,r}$ and, thus,
$\eta_{t+r}(x)\ge \eta_{t,r}^*(x)$, 
	\begin{align}
		\E\left[\left\|\eta_{t+r}-\eta_{t,r}^*\right\|_1\right] 
		& = \E\left[\left\|\eta_{t+r} \right\|_1 \right]- \E\left[\|\eta_{t,r}^*  \|_1 \right]
		\\
		& = \E\left[\left\|\eta_{t} \right\|_1 \right]- \E\left[\|\eta_{t}^*  \|_1 \right]=
		\E\left[\|\widehat\eta_t  \|_1 \right]\,.
	\label{eq:pezzo1}
	\end{align}
	The same argument shows that 
	\begin{equation}\label{eq:pezzo3}
	\E\left[\left\|\eta_{t,\infty}^*-\eta_\infty\right\|_1\right]
=\E\left[\|\widehat\eta_t  \|_1 \right]\,.
	\end{equation}
	Concerning the second term on the right-hand side of \eqref{eq:three}, using Cauchy-Schwarz and 
	 the $L^2$ contraction from the proof of Lemma \ref{lemma:AL} (Section \ref{sec:proof-L2-contraction}), one has 
	\begin{equation}\label{eq:l2bo}
		\E\left[\left\|\eta_{t,r}^*-\eta_{t,\infty}^*\right\|_1\right]
		\le \sqrt n \,e^{-\frac12\lambda \,r}\,\E\left[\left\|\eta_{t}^*-\eta_{t,\infty}^*\right\|^2_2\right]^{\frac12}\,,
	\end{equation}
	where $\l=\l_{\rm SRM}\sim c_1/n$ for some constant $c_1>0$, and we have used the $L^2$ contraction
	with initial condition $\eta_{t}^*$. From \eqref{eq:L2-norm-eq} and \eqref{eq:tous} it follows that 
\begin{equation}
\E\left[\left\|\eta_{t}^*-\eta_{t,\infty}^*\right\|^2_2\right]	
\le \E\big[\|\eta_{t}^*\|_2^2\big]+\E\big[\|\eta_{t,\infty}^*\|_2^2\big]\le  \E\big[\|\eta_{t}^*\|_2^2\big]+\E\big[\|\eta_{t,\infty}\|_2^2\big]\le n^{-1}(e^{6\psi}+O(1))\,.
\end{equation}
Since $r=c\gamma n\sqrt{\log n}$, if $c$ is larger than some constant independent of $\g$ and $n$, then $\l r \ge 8\psi$ and \eqref{eq:l2bo} yields
\begin{equation}\label{eq:pezzo2}
		\E\left[\left\|\eta_{t,r}^*-\eta_{t,\infty}^*\right\|_1\right]
		 = O(e^{-\psi})\,.
	\end{equation}
Summarizing, the desired result \eqref{eq:eta+22}  will follow from \eqref{eq:three}, \eqref{eq:pezzo1}, \eqref{eq:pezzo3}, \eqref{eq:pezzo2} and the following lemma, which we state and prove for all three models.  
	\begin{lemma}\label{lemma:eta*B}
	All three models satisfy
		\begin{equation}\label{eq:topro}
			\lim_{n\to\infty}\big(\E[\|\widehat\eta_{t}\|_1]-\E[\|\eta^+_{t}\|_1]\big)=0\,.
		\end{equation}
	\end{lemma}
	
Before giving the proof of the lemma, we observe that it concludes the proof of the upper bound for both  the SRM and the GAM. Indeed,  the above argument may be repeated without modifications for the GAM, with the matrices $R_i$ replaced by $Q_i$ as in \eqref{eq:repres}, and with $\l=\l_{\rm GAM}$  instead of $\l_{\rm SRM}$ in \eqref{eq:l2bo}. To complete the proof for the SEM there are only a few minor adjustments to be made. Indeed, the estimate \eqref{eq:l2bo} is not available here, and we argue instead using Lemma \ref{lemma:AL_HMP}. 
	Namely, by using the coupling \eqref{eq:GC}, and the fact that $\eta_{t,\infty}^*$ is a flat configuration, one finds
	\begin{align}
	\E\left[\left\|\eta_{t,r}^*-\eta_{t,\infty}^*\right\|_2^2\right]	
		& = \E[\|\eta_{t,r}^*-\langle \eta_{t,r}^*\rangle+\langle \eta_{t,r}^*\rangle-\eta_{t,\infty}^*\|_2^2] \\
		&= \E[\|\eta_{t,r}^*-\langle \eta_{t,r}^*\rangle\|_2^2] + n\E[(\langle \eta_{t,r}^*\rangle-\langle \eta_{t,\infty}^*\rangle)^2]\\
		&\le e^{-\l_{\rm SEM}r}\E[\|\eta_t^*-\langle \eta_t^*\rangle\|_2^2]  + n\E[\langle \eta_{t,\infty}^*\rangle^2-\langle \eta_{t,r}^*\rangle^2]\,,
	\end{align}
	where we have used Lemma \ref{lemma:AL_HMP}. To estimate the last term above, recall \eqref{eq:sipo} and note that
	for any $s>0$ we may estimate
	\[
	\E[\langle \eta_{t,r+s}^*\rangle^2-\langle \eta_{t,r+s-1}^*\rangle^2]\le \frac{a_n}{n^2}\,(1-\l_{\rm SEM})^{s}\E[\|\eta_{t,r}^*-\langle \eta_{t,r}^*\rangle\|_2^2]\,,
	\]
	where $a_n=\frac{2(4\E[X^2]-1)}{n-1}$.
	Summing over $s$, we obtain
	\[n\E[\langle \eta_{t,\infty}^*\rangle^2-\langle \eta_{t,r}^*\rangle^2]\le \frac{a_n}{n\l_{\rm SEM}}\E[\|\eta_{t,r}^*-\langle \eta_{t,r}^*\rangle\|_2^2]\le  \frac{c_1}n\,e^{-\l_{\rm SEM}r}\E[\|\eta_t^*-\langle \eta_t^*\rangle\|_2^2]\,,
	\]
	where $c_1>0$ is a constant independent of $n$, since $\frac{a_n}{n\l_{\rm SEM}}=O(1/n)$, and we have used again Lemma \ref{lemma:AL_HMP}. From \eqref{eq:tous} we conclude that
	\begin{equation}
		\begin{split}
	\E\left[\left\|\eta_{t,r}^*-\eta_{t,\infty}^*\right\|_2^2\right]	
	&\le \tonde{1+\frac{c_1}{n}}e^{-\lambda_{\rm SEM}r}\E\big[\|\eta_t^*-\langle \eta_t^*\rangle\|_2^2\big]\\
		&\le  \tonde{1+\frac{c_1}{n}}e^{-\lambda_{\rm SEM}r}\E\big[\|\eta_t^*\|_2^2\big]\le  \tonde{1+\frac{c_1}{n}}\,n^{-1}e^{-\lambda_{\rm SEM}r+6\psi}
		\,.
		\end{split}
	\end{equation}
	Hence, by taking  $r=c\g n\sqrt {\log n}$ with $c$ a constant independent of $n$ and $\g$, such that $r\lambda_{\rm SEM}\ge 8\psi$, we deduce that 
\begin{equation}\label{eq:l2bo2}
		\E\left[\left\|\eta_{t,r}^*-\eta_{t,\infty}^*\right\|_1\right]
		\le 
		\sqrt n \,\E\left[\left\|\eta_{t,r}^*-\eta_{t,\infty}^*\right\|_2^2\right]^{\frac12}
		=O(e^{-\psi})\,.
	\end{equation}
	This concludes the argument for SEM as well. It remains to prove Lemma \ref{lemma:eta*B}.

\begin{proof}[Proof of Lemma \ref{lemma:eta*B}]	Recall the definitions \eqref{eq:def-eta+g} and \eqref{eq:def-eta-star}, and define
	\begin{equation}
		B_t=\{x:\, n\eta_{t}^-(x)\ge e^{3\psi} \}\,.
	\end{equation}
	In this way \eqref{eq:topro} reduces to 
		\begin{equation}\label{eq:topros0}
			\lim_{n\to\infty}\E\left[\eta_{t}^-(B_t)\right]=0\,,
		\end{equation}
		where we use the notation 
		\begin{equation}\label{eq:zeta-B}
				\eta_t^-(B_t)=\sum_{x\in B_t}\eta^-_t(x) = \sum_{x\in B_t}\sum_{\zeta\in A_t(x)}|\zeta|\car_{n|\zeta|<e^\psi}
				\,.
				\end{equation}
We will actually prove that \eqref{eq:topros0} holds uniformly in $t\ge 0$, i.e.\
\begin{equation}\label{eq:topros}
			\lim_{n\to\infty}\sup_{t\ge 0}\E\left[\eta_{t}^-(B_t)\right]=0\,,
		\end{equation}
Notice that
		\begin{align}
				\eta_t^-(B_t)&=\sum_{x}\car_{n\eta_t^-(x)\ge e^{3\psi}}\sum_{\zeta\in A_t(x)}|\zeta|\car_{n|\zeta|< e^{\psi}}\\
				&\le ne^{-3\psi}\sum_{x}\eta_t^-(x)\sum_{\zeta\in A_t(x)}|\zeta|\car_{n|\zeta|< e^{\psi}}
				= ne^{-3\psi}\sum_{x}\eta_t^-(x)^2\,.
		\label{eq:zeta-B2}\end{align}

		We now estimate  the expectation of the sum in the right hand side of \eqref{eq:zeta-B2}. One has 
		\begin{align}\label{eq:Ed-End}
	\eta_t^-(x)^2=\sum_{\zeta,\zeta'\in A_t(x)}|\zeta|\,|\zeta'|\car_{n|\zeta|<{e^\psi}}\car_{n|\zeta'|<{e^\psi}}\le E_t^{\rm d}(x)+E_t^{\rm nd}(x)\,,
			\end{align}
	where we define 
		\begin{align}
	E_t^{\rm d}(x)=\sum_{\zeta\in A_t(x)}|\zeta|^2\car_{n|\zeta|<{e^\psi}}
	\,,\qquad 
	E_t^{\rm nd}(x) = \sum_{\zeta,\zeta'\in A_t(x)}|\zeta|\,|\zeta'|\car_{\zeta\neq \zeta'}\,,
			\end{align}
		The expectation of the sum of the diagonal terms $E_t^{\rm d}(x)$ can be bounded as
		\begin{equation}\label{eq:zeta-B3}
			\sum_{x}\E[E_t^{\rm d}(x)]\le \frac{e^\psi}{n}\sum_{x}\E\left[ \sum_{\zeta\in A_t(x)}|\zeta|\car_{|\zeta|<\frac{e^\psi}{n}}\right]=\frac{e^\psi}{n}\E[\|\eta_t^-\|_1]\le \frac{e^\psi}{n}\,.
		\end{equation}
			For the non-diagonal terms $E_t^{\rm nd}(x)$, we argue as follows.  Let $(u,v)$ denote the random pair of particles to be updated at time $t$ and let $X$ denote the associated redistribution variable. If $x\notin (u,v)$ then nothing happens at $x$ and therefore $A_t(x) = A_{t-1}(x)$ in this case. It follows that
			 \begin{equation}\label{eq:xnot}
			\E\left[\car_{x\notin (u,v)}E_t^{\rm nd}(x)\right] = \left(1-\frac2n\right)\E\left[E^{\rm nd}_{t-1}(x)\right]\,.
					\end{equation}
					If $(u,v) = (x,y)$, then any pile $\zeta\in A_t(x)$ has size $X|\zeta_0|$ where $\zeta_0$ is a uniquely determined pile, the parent, that was either in $A_{t-1}(x)$ or $A_{t-1}(y)$. This is true for the SRM. 	We will stick with the case of the SRM for the sake of simplicity, and will comment on the minor modifications needed in the other two models later on. We note that in any case $\zeta\neq\zeta'$ implies $\zeta_0\neq\zeta_0'$. Therefore, for a fixed pair $(x,y)$,
			 \begin{align}\label{eq:xin}
			\E\left[\car_{(x,y)= (u,v)}E_t^{\rm nd}(x)\right]& =\frac{\E[X^2]}{n(n-1)}\E\left[\sum_{\zeta,\zeta'\in A_{t-1}(x)\cup A_{t-1}(y)}|\zeta|\,|\zeta'|\car_{\zeta\neq \zeta'}
			\right]\,. 
					\end{align}
					The last expectation above can be rewritten as 
					 \begin{align}\label{eq:xin2}
					 \E\left[\sum_{\zeta,\zeta'\in A_{t-1}(x)\cup A_{t-1}(y)}|\zeta|\,|\zeta'|\car_{\zeta\neq \zeta'}
			\right]= \E[E_{t-1}^{\rm nd}(x)] + \E[E_{t-1}^{\rm nd}(y)]+2\E[\eta_{t-1}(x)\eta_{t-1}(y)]\,.
				\end{align}	
			The case $(u,v)=(y,x)$ is analogous, with $X$ replaced by $1-X$. Since $\E[X^2]=\E[(1-X)^2]$, the expressions \eqref{eq:xnot} and \eqref{eq:xin} imply
		\begin{align}
			\sum_{x}\E[E_t^{\rm nd}(x)]&\le  \left(1-\frac2n\right)\sum_{x}\E\left[E^{\rm nd}_{t-1}(x)\right] +\\& \quad + \frac{2\E[X^2]}{n(n-1)}{\sum_{x,y\,:\,x\neq y}}\left(\E[E_{t-1}^{\rm nd}(x)] + \E[E_{t-1}^{\rm nd}(y)]+2\E[\eta_{t-1}(x)\eta_{t-1}(y)]\right)\\ & 
			\leq \left(1-\frac2n + \frac{4\E[X^2]}n\right)\sum_{x}\E\left[E^{\rm nd}_{t-1}(x)\right] + \frac{4\E[X^2]}{n(n-1)}\,\E\left[\|\eta_{t-1}\|_1^2\right]\,.
					\end{align}
		The above estimate has been derived for SRM. Let us remark that the same expression holds for SEM and GAM with the only modification that 	the last term in the right hand side has $\E[X^2]$ replaced by $\E[X(1-X)]$. This follows from symmetry of $X$ and the fact that 
		in GAM and SEM one has $\zeta=X|\zeta_0|$ if $\zeta_0\in A_{t-1}(x)$ and $\zeta=(1-X)|\zeta_0|$ if $\zeta_0\in A_{t-1}(y)$.
Moreover, for SRM and GAM		one has $\|\eta_{t}\|_1=1$ for all $t$. For SEM, using \eqref{eq:dual} it follows that $\E\left[\|\eta_t\|_1^2\right]=O(1)$. In conclusion, we see that, for an absolute constant $C>0$, for all three models one has
		\begin{equation}
			\sum_{x}\E[E_t^{\rm nd}(x)]
				\le\frac{C}{n^2}+\left(1-\frac{2-4\E[X^2]}n\right)\sum_{x}\E[E_{t-1}^{\rm nd}(x)]\,.
		\end{equation}
		Iterating,	and using $\sum_xE_{0}^{\rm nd}	(x)=0$, 
		\begin{equation}\label{eq:zeta-B4}
			\sum_{x}\E[E_t^{\rm nd}]\le \frac{C}{n^2}\sum_{s=0}^\infty\left(1-\frac{2-4\E[X^2]}n\right)^s= \frac{C}{n(2-4\E[X^2])}=O\bigg(\frac1n\bigg)\,, 
		\end{equation}
		where $2-4\E[X^2]>0$ holds  by the non-degeneracy assumption. Remark that  the obtained estimate is uniform in $t\ge0$. By \eqref{eq:zeta-B2}, \eqref{eq:Ed-End}, \eqref{eq:zeta-B3} and \eqref{eq:zeta-B4},  we get \eqref{eq:topros}, as desired.
	\end{proof}

	\begin{acknowledgement}
	 MQ and FS are members of GNAMPA, INdAM, and acknowledge financial support through the GNAMPA project \textquotedblleft Redistribution models on networks\textquotedblright.
	\end{acknowledgement}
	
%	\bibliographystyle{alpha}
%	\bibliography{bookshelf}

\newcommand{\etalchar}[1]{$^{#1}$}

\end{document}